\documentclass{amsart}

\usepackage{amsmath}
\usepackage{amssymb}
\usepackage{amsthm}

\newtheorem{theorem}{Theorem}[section] 
\newtheorem{proposition}[theorem]{Proposition} 
\newtheorem{lemma}[theorem]{Lemma} 

\theoremstyle{definition}
\newtheorem{definition}[theorem]{Definition}

\newtheorem{defprop}[theorem]{Definition-Proposition}
\theoremstyle{remark} 
\newtheorem{remark}[theorem]{Remark}

\newcommand{\cf}{\emph{cf.}~}
\newcommand{\ie}{\emph{i.e.}~}

\renewcommand{\and}{\text{and}} 
\renewcommand{\for}{\text{ for }}
\renewcommand{\forall}{\text{ for all }}

\newcommand{\ZZ}{\mathbb{Z}}
\newcommand{\RR}{\mathbb{R}}
\newcommand{\CC}{\mathbb{C}}

\newcommand{\ra}{\rightarrow}

\newcommand{\q}{q:\ZZ^n\rightarrow\ZZ}
\newcommand{\B}{(-,-)}

\newcommand{\Gq}{G}
\newcommand{\tildeG}{\tilde{G}}

\renewcommand{\H}{\tilde{H}}

\newcommand{\set}[1]{\{#1\}}
\newcommand{\setP}[2]{\set{#1\mid#2}}

\DeclareMathOperator{\ad}{ad}
\DeclareMathOperator{\rad}{rad} 
\DeclareMathOperator{\radC}{rad_{\CC}} 
\DeclareMathOperator{\rank}{rank} 
\DeclareMathOperator{\corank}{corank} 

\begin{document}

\title[The EALA associated with a connected non-negative unit
form]{The extended affine Lie algebra associated with a connected
  non-negative unit form}

\author[G. Jasso]{Gustavo Jasso}
\email{jasso.ahuja.gustavo@b.mbox.nagoya-u.ac.j}

\address{Instituto de Matem\'aticas, Univerisdad Nacional
  Aut\'otonoma de M\'exico. Mexico City, Mexico}

\thanks{The author wishes to acknowledge the support of
  M. Barot, who directed the author's undergraduate
  thesis \cite{ME} from which this article is derived. This acknowledgment
  is extended to L. Demonet for his helpful comments and
  suggestions on an earlier version of this article which greatly
  improved its readability.}  



\begin{abstract}
  Given a connected non-negative unit form we construct an
  extended affine Lie algebra by giving a Chevalley basis for
  it. 
  We also obtain this algebra as a quotient of an algebra
  defined by means of generalized Serre relations by M. Barot,
  D. Kussin and H. Lenzing. 
  This is done in an analogous way to the construction of the
  simply-laced affine Kac-Moody algebras.
  Thus, we obtain a family of extended affine Lie algebras of
  simply-laced Dynkin type and arbitrary nullity. 
  Furthermore, there is a one-to-one correspondence between these
  Lie algebras and the equivalence classes of connected
  non-negative unit forms.
\end{abstract}

\maketitle

\section{Introduction}

To each unit form $\q$, M. Barot, D. Kussin and H. Lenzing
associate in \cite{BKL} a complex Lie algebra $\tildeG(q)$ by
means of generalized Serre relations.
We are interested in the case when the form $q$ is connected and
non-negative, see Section \ref{sec:ealas} for definitions.
When the rank of the radical of $q$ is 0 or 1,
the algebra $\tildeG(q)$ is a simply-laced finite
dimensional simple Lie algebra or a simply-laced affine
Kac-Moody algebra respectively.
Also, the algebra $\tildeG(q)$ admits a root space decomposition
with respect to a finite dimensional abelian Cartan subalgebra
$\H$ with the roots of $q$, \ie the set 
$R(q):=q^{-1}(0)\cup q^{-1}(1)$, as root system.
Moreover, $R(q)$ is a so-called extended affine root system, 
see Definition \ref{def:ears} and Proposition \ref{prop:Rq}.
Extended affine root systems are a certain generalization of
affine root systems; one important difference is that, in general,
isotropic roots generate a subspace (of an euclidean space, say) 
of dimension higher than 1.
Extended affine root systems are precisely the root systems
associated with extended affine Lie algebras. Roughly speaking,
extend affine Lie algebras are characterized by the existence of a
non-degenerate symmetric invariant bilinear form which induces a
root-space decomposition with respect to a finite dimensional
abelian subalgebra; moreover, the adjoint representations of
homogeneous elements of non-isotropic degree are locally
nilpotent, see Definition \ref{def:eala}.
Examples of extended affine Lie algebras are both the
simple finite-dimensional Lie algebras and the affine Kac-Moody
algebras.
We refer the reader to \cite{AABGP} for an introductory treatment
of the theory of extended affine Lie algebras and motivation for
their study.

Our aim is to study the connection between extended affine Lie
algebras and the Lie algebra $\tildeG(q)$ described above.
In fact, the Lie algebra $\tildeG(q)$ is very close to
being an extended affine Lie algebra, but in general it lacks an
invariant non-degenerate symmetric bilinear form (the
non-degeneracy being the main obstruction).
We can obtain such a form by passing to a natural quotient of
$\tildeG(q)$.
Thus, this article can be regarded as a simple extension of
\cite{BKL} towards the theory of extended affine Lie algebras.

Let us explain the contents of this article.
In Section \ref{sec:ealas}, we recall the definitions of an
extended affine Lie algebra and of an extended affine root
system. 
We also show how to associate an extended affine root system to a
given connected non-negative unit form.
In Section \ref{sec:eala_associated_unit_form}, given a connected
non-negative unit form $q$, we construct a Lie algebra $E(q)$
with root system $R(q)$ and with nullity the corank of $q$.
This construction is a slight modification of the construction of
\cite[Sec. 2]{BKL}, which also is related to a construction given by
Borcherds in \cite{Borcherds}.
The modification concerns both the Cartan subalgebra and the root
spaces associated to non-isotropic roots. 
Our main result is the following:

\begin{theorem}[see Theorems \ref{thm:Eq-is-EALA}
  and \ref{thm:q-q'}] 
  \label{thm:Eq}
  Let $\q$ be a connected non-negative unit form with associated
  root system $R(q)$. 
  Then the Lie algebra $E(q)$ is an extended affine Lie algebra
  with root system $R(q)$. 
  Furthermore, let $q$ and $q'$ be connected non-negative unit
  forms. Then, $q$ and $q'$ are equivalent if and only if $E(q)$ and
  $E(q')$ are isomorphic as graded Lie algebras.
\end{theorem}

In Section \ref{sec:kacmoody_construction}, we recall from
\cite{BKL} the construction of the Lie algebra $\tildeG(q)$ by
means of generalized Serre relations. 
We show that $E(q)$ can also be obtained from $\tildeG(q)$ in a
way that imitates the construction  of the affine Kac-Moody
algebras, \cf \cite[Chapter 14]{Carter}.  
More precisely, we have the following result:

\begin{theorem}\label{thm:GKM}
  Let $\q$ be a connected non-negative unit form. 
  Then there is an isomorphism of graded Lie algebras between
  $\tildeG(q)/I$ and $E(q)$ where $I$ is the unique
  maximal ideal of $\tildeG(q)$ with $I\cap \H=\set{0}$.
\end{theorem}

\section{Preliminaries}
\label{sec:ealas}

We recall the definitions of an extended affine Lie algebra and
of an extended affine root system.
We follow closely the exposition of \cite{AABGP}.
Afterwards, we recall basic concepts regarding unit forms and
present a straightforward way to associate an extended affine root
system with a connected non-negative unit form, 
see Proposition \ref{prop:Rq}.

\subsection{Extended affine Lie algebras and their root systems}

Let $L$ be a Lie algebra over the field $\CC$ of complex numbers
having the following properties:
\begin{enumerate}
\item[(EA1)] $L$ has a non-degenerate symmetric bilinear form
  \[ 
  \B:L\times L\rightarrow\CC 
  \]
  which is invariant in the sense that 
  \[
  ([x,y],z)=(x,[y,z])\forall x,y,z\in L. 
  \]
\item[(EA2)] $L$ has a non-trivial finite dimensional abelian
  subalgebra $H$ which equals its own centralizer in $L$ and
  such that $\ad_{L}h$ is diagonalizable for all $h\in H$.
\end{enumerate}
It follows from (EA2) that $L$ admits a root space decomposition 
\[
L=\bigoplus_{\alpha\in H^*} L_{\alpha},
\]
where, for each $\alpha\in H^*$,
\[
L_{\alpha}:=\setP{x\in L}{[h,x]=
  \alpha(h)x \quad\text{for all}\quad h\in H}.
\]
We call 
\[
R=\setP{\alpha\in H^*}{L_{\alpha}\neq\{0\}}
\]
the \emph{root system of $L$ with respect to $H$} (note that
$(L_{\alpha},L_{\beta})=\{0\}$ whenever $\alpha+\beta\neq 0$). 
Since the form $\B$ is invariant and $\ad_L h$ is diagonalizable
for all $h\in H$, the restriction of the form to $H$
is non-degenerate.
Thus we may transfer the bilinear form $\B$ to $H^*$ via the induced isomorphism
\begin{align*}
  H &\rightarrow H^*\\
  h &\mapsto (h,-).
\end{align*}
Then, we let 
\[
R^0=\setP{\alpha\in R}{(\alpha,\alpha)=0}
\quad\text{and}\quad
R^\times=R\setminus R^0. 
\]
We call $R^0$ the set of \emph{isotropic roots} and $R^\times$
the set of \emph{non-isotropic roots}.
\begin{enumerate}
\item[(EA3)] If $\alpha\in R^\times$ is a non-isotropic root and
  $x_{\alpha}\in L_{\alpha}$, then $\ad_{L}x_{\alpha}$ acts
  locally nilpotently on $L$.
  That is, for each $y\in L$ there exists $n>0$ such that 
  \[
  (\ad_{L}x_{\alpha})^n y=0.
  \]
\item[(EA4)] $R$ is a discrete subset of $H^*$.
\item[(EA5)] $L$ is irreducible, \ie~$R^\times$ cannot be
  decomposed as a disjoint union $R_{1}\amalg R_{2}$ where
  $R_{1},R_{2}$ are non-empty subsets of $R^\times$ satisfying
  $(R_{1},R_{2})=\set{0}$.
\end{enumerate}

\begin{definition}
  \label{def:eala}
  \cite[Ch. I]{AABGP}
  Let $L$ be
  a Lie algebra over $\CC$ such that it satisfies the five
  properties (EA1)-(EA5). Then, the triple $(L,H,\B)$ is called
  an \emph{extended affine Lie algebra}.
\end{definition}

The root systems of extended affine Lie algebras can be axiomatically described.
They belong to the class of extended affine root systems.

\begin{definition}
  \label{def:ears}
  \cite[Ch. II, Def. 2.1]{AABGP}
  Assume that $V$ is a finite dimensional real vector space with
  a non-negative symmetric bilinear form $\B$, and that $R$ is a
  subset of $V$. 
  Let 
  \[ 
  R^0=\setP{\alpha\in R}{(\alpha,\alpha)=0} 
  \quad \text{and}\quad
  R^\times=\setP{\alpha\in R}{(\alpha,\alpha)\neq 0}.
  \] 
  We call $R$ an \emph{extended affine root system}
  in $V$ if $R$ has the following properties:
  \begin{enumerate}
  \item[(R1)] $0\in R$. 
  \item[(R2)] $-R=R$. 
  \item[(R3)] $R$ spans $V$. 
  \item[(R4)] If $\alpha\in R^\times$ then $2\alpha\notin R$. 
  \item[(R5)] $R$ is discrete in $V$. 
  \item[(R6)] If $\alpha\in R^\times$ and $\beta\in R$ then there
    exist $d,u\in \ZZ_{\geqslant0}$ such that
    \[
    \{\beta+n\alpha:n\in\ZZ\}\cap
    R=\{\beta-d\alpha,\dots,\beta+u\alpha\}
    \quad\text{and}\quad
    d-u=2\frac{(\alpha,\beta)}{(\alpha,\alpha)}.
    \]
  \item[(R7)] $R^\times$ cannot be decomposed as a disjoint union
    $R_{1}\amalg R_{2}$ where $R_{1},R_{2}$ are non-empty subsets
    of $R^\times$ satisfying $(R_{1},R_{2})=\set{0}$.
  \item[(R8)] For any $\sigma\in R^0$ there exists $\alpha\in
    R^\times$ such that $\alpha+\sigma\in R$. 
  \end{enumerate}
\end{definition}

\subsection{Unit forms}

A \emph{unit form} is a quadratic form $\q$ of the form
\[ 
q(x)=\sum_{i=1}^n x_{i}^2+\sum_{i<j} q_{ij}x_{i}x_{j} 
\for
x\in\ZZ^n, 
\]
such that $q_{ij}\in\ZZ$ for all $1\leqslant i<j\leqslant n$.
We say that $q$ is \emph{positive} (resp. \emph{non-negative}) if
$q(x)>0$ (resp. $q(x)\geqslant0$) for all
$x\in\ZZ^n\setminus\{0\}$. 
Recall that any unit form has an
\emph{associated symmetric matrix} $C=C_{q}\in\ZZ^{n\times n}$
defined by
\[ 
C_{ij}:=q(c_i+c_j)-q(c_i)-q(c_j), 
\]
where $\set{c_1,\dots,c_n}$ is the standard basis of $\ZZ^n$. 
Also, recall that 
\[ 
q(x,y):=x^\top Cy,\for x,y\in\ZZ^n
\] 
is the \emph{symmetric bilinear form associated with $q$} and
we have $2q(x)=q(x,x)$ for each $x\in\ZZ^n$.
The \emph{radical} of $q$ is the subgroup of $\ZZ^n$ defined
by
\[
\rad q := \setP{x\in\ZZ^n}{q(x,-)=0}.
\]
Further, we set $\corank q:=\rank(\rad q)$. Observe that $\rad
q\subseteq q^{-1}(0)$. 
If $q$ is non-negative, then $\rad q=q^{-1}(0)$. 
The \emph{associated bigraph of $q$} is the graph with vertices
$1,\dots,n$ and $|q_{ij}|$ solid edges between vertices $i$ and
$j$ if $q_{ij}<0$ and $|q_{ij}|$ dotted edges
between $i$ and $j$ if $q_{ij}>0$. We say that a unit form is
\emph{connected} if its associated bigraph is connected. 
We say that two non-negative unit forms $q$ and $q'$ are
\emph{equivalent} if there exist an automorphism $T$ of $\ZZ^n$
such that $q'=q\circ T$. 
Connected non-negative unit forms can be classified as follows:

\begin{theorem}
  \label{thm:Dynkin-types}
  \cite[Thm. (Dynkin-types)]{BdP}
  Let $\q$ be a non-negative unit form.
  Then there exists a $\ZZ$-invertible transformation
  $T:\ZZ^n\to\ZZ^n$ such that
  \[
  (q\circ T)(x_1,\dots,x_n) = q_{\Delta}(x_1,\dots,x_{n-\nu})
  \]
  where $\nu = \corank q$ and $\Delta$ is a Dynkin diagram with
  $n-\nu$ vertices
  uniquely determined by $q$.
  We call $\Delta$ the Dynkin type of $q$.
\end{theorem}

Note that Theorem \ref{thm:Dynkin-types} implies that two
connected non-negative unit forms are equivalent if and only if
they have the same corank and the same Dynkin type, see the corollary
after \cite[Thm. (Dynkin-types)]{BdP}.

There is a straightforward way to associate an extended affine root system to a given
connected non-negative unit form.

\begin{defprop}
  \label{prop:Rq}
  Let $\q$ be a connected non-negative unit form. 
  Then 
  \[
  R(q):=q^{-1}(0)\cup q^{-1}(1)\subset \ZZ^n.
  \]  
  is an extended affine root system in $\RR^n$ with respect to the symmetric bilinear
  form of $q$.
\end{defprop}

\begin{proof}
  All but properties (R6) and (R7) are easily verified (for property
  (R3) observe that the standard basis of $\ZZ^n$ belongs to
  $R(q)$ since $q(c_i)=1$ for all $i\in\set{1,\dots,n}$).
  We first show that $R(q)$ verifies axiom (R6). 
  Let $\alpha\in q^{-1}(1)$, $\beta\in R(q)$, and $n\in\ZZ$. 
  We need to determine the values of $n$ for which
  $q(\beta+n\alpha)\in\set{0,1}$. 
  Note that we have
  \begin{align*}
    q(\beta+n\alpha)=&\tfrac{1}{2}q(\beta+n\alpha,\beta+n\alpha)\\
    =& \tfrac{1}{2}(q(\beta,\beta) + 2nq(\alpha,\beta) 
    + n^2q(\alpha,\alpha))\\
    =& q(\beta) + nq(\alpha,\beta) + n^2.
  \end{align*}
  First, suppose that $q(\beta)=0$. Since $q$ is non-negative we
  have that $\beta\in\rad q$ and thus $q(\alpha,\beta)=0$.
  Then we have that $\beta+n\alpha\in
  R(q)$ if and only if $n\in\set{-1,0,1}$. Thus, taking $d=u=1$
  gives the result in this case.
  Now suppose otherwise that $q(\beta)=1$. 
  A simple calculation shows that $q(\beta+n\alpha)=1$ if and only
  if $n=0$ or $n=-q(\alpha,\beta)$.
  On the other hand, $q(\beta+n\alpha)=0$ if and only if
  \[
  n = \frac{-q(\alpha,\beta)\pm\sqrt{q(\alpha,\beta)^2-4}}{2}
  \]
  Since $q(\beta+n\alpha)\geqslant 0$, by putting
  $x=-\tfrac{1}{2}q(\alpha,\beta)$ and evaluating
  $q(\beta+x\alpha)$ we obtain that $-2\leqslant
  q(\alpha,\beta)\leqslant 2$.
  Hence $q(\beta+n\alpha)=0$ has
  integer solution if and only if
  $q(\alpha,\beta)\in\set{-2,2}$, which is given by
  $n=-\tfrac{1}{2}q(\alpha,\beta)$.
  Thus our analysis shows that $q(\beta+n\alpha)\in R(q)$ if and
  only if
  \[ 
  \begin{cases}
    n\in\set{0,-q(\alpha,\beta)}
    &\text{if } q(\alpha,\beta)\in\set{-1,1}\\
    n\in\set{0,-\tfrac{1}{2}q(\alpha,\beta),-q(\alpha,\beta)}
    &\text{if }q(\alpha,\beta)\in\set{-2,2}.
  \end{cases}
  \]
  Since
  $q(\alpha,\beta)=\tfrac{2q(\alpha,\beta)}{q(\alpha,\alpha)}$
  the claim follows in this case also.

  We now show that $R(q)$ satisfies axiom (R7). 
  Let $R(q)^{\times}=R_1\amalg R_2$.
  Without loss of generality, we can assume using the
  connectedness of $q$ that $\{c_1,\dots,c_n\}\subset R_1$.
  On the other hand, for any $\alpha\in R_2$ there exist
  $\alpha_i\in\ZZ$ such that $\alpha=\sum_{i=1}^n\alpha_i c_i$.
  Thus we have
  \[
  q(\alpha,\alpha)=\sum_{i=0}^n\alpha_iq(c_i,\alpha)
  \quad\text{for all}\quad
  \alpha\in R_2\subset R(q)^{\times}.
  \]
  Hence if $(R_1,R_2)=\set{0}$, then $R_2$ must be empty.
  This shows that $R(q)$ is irreducible.
  Thus $R(q)$ satisfies properties (R1)-(R8) and is thus an
  extended affine root system.\qedhere 
\end{proof}

\section{Main result}
\label{sec:eala_associated_unit_form}

To a connected non-negative unit form $q$, we associate an
extended affine Lie algebra $E$ by modifying the construction of
\cite[Sec. 2]{BKL}.
The modification concerns the root
spaces attached to non-zero isotropic roots. 
We give the proof of Theorem \ref{thm:Eq} in the second part of
this section. 

\subsection{Construction}

We fix a connected non-negative unit form $\q$ with symmetric
matrix $C\in\ZZ^{n\times n}$.
For simplicity, we let
\[
R:= R(q),
\quad
R^\times := R(q)^\times = q^{-1}(1)
\quad \text{and}\quad
R^0 := R(q)^0 = q^{-1}(0).
\]
Recall that since $q$ is non-negative we have that
$R^0=q^{-1}(0)=\rad q$.
Let $\radC q:=\CC\otimes_{\ZZ}\rad q$ and define complex vector
spaces $N_\alpha=N(q)_\alpha$ for $\alpha\in R$ by 
\[ 
N_\alpha:=\begin{cases}
  \CC e_{\alpha} &\text{if } \alpha\in R^\times,\\
  \CC^n/\CC\alpha &\text{if } \alpha\in R^0\setminus\{0\},\\
  \CC^n\oplus(\radC q)^* &\text{if } \alpha=0,
\end{cases}
\]
where $(\radC q)^*$ denotes the $\CC$-dual of $\radC q$ and define
\[
N=N(q):=\bigoplus_{\alpha\in R} N_\alpha 
\]
as a vector space. We can endow $N$ with a Lie algebra
structure as follows: 
Let $\pi_{\alpha}:\CC^n\rightarrow N_\alpha$ be the
canonical projection for each non-zero $\alpha\in R^0$; by
convention, $\pi_0(v)=v$ for all $v\in\CC^n$. 
Choose a non-symmetric bilinear form $B:\ZZ^n\times\ZZ^n\ra\ZZ$
such that $q(x)=B(x,x)$ for all $x\in\ZZ^n$ and set
$\epsilon(\alpha,\beta)=(-1)^{B(\alpha,\beta)}$. 
Finally, choose a projection $\rho:\CC^n\to\radC q$ and
define the following bracket rules that depend on the choice of
$B$ and $\rho$.

Let $\alpha,\beta\in R^\times$, $\sigma,\tau\in R^0$,
$v,w\in\CC^n$ and $\xi,\zeta\in(\radC q)^*$:
\begin{align}
  \tag{B1}
  [\pi_{\sigma}(v),\pi_{\tau}(w)] &=
  \epsilon(\sigma,\tau)q(v,w)\pi_{\sigma+\tau}(\sigma)\\
  \tag{B2} 
  [\pi_{\sigma}(v),e_{\beta}] &= -[e_{\beta},\pi_{\sigma}(v)] = 
  \epsilon(\sigma,\beta)q(v,\beta)e_{\beta+\sigma}.\\
  \tag{B3}
  [e_{\alpha},e_{\beta}] &= 
  \begin{cases}
    \epsilon(\alpha,\beta)e_{\alpha+\beta}& \text{if
    }\alpha+\beta \in R^\times,\\ 
    \epsilon(\alpha,\beta)\pi_{\alpha+\beta}(\alpha)& \text{if
    }\alpha+\beta\in R^0,\\ 
    0 &\text{otherwise.} 
  \end{cases}\\
  \tag{B4} 
  [\xi,e_{\beta}]&=-[e_{\beta},\xi] = \xi(\rho(\beta))e_{\beta}.\\ 
  \tag{B5} 
  [\xi,\pi_{\tau}(w)]&= -[\pi_{\tau}(w),\xi] =
  \xi(\rho(\tau))\pi_{\tau}(w).\\ 
  \tag{B6}
  [\xi,\zeta]&=0.
\end{align}

\begin{remark}
  \label{rmk:anticommutativity}
  If $\alpha,\beta\in R^\times$ are such that
  $\alpha+\beta \in R^0 = \rad q$, then
  $\pi_{\alpha+\beta}(\beta)=-\pi_{\alpha+\beta}(\alpha)$.
  Hence the above bracket is anticommutative.
\end{remark}

\begin{proposition}
  The rules (B1)-(B6) define a Lie algebra structure on
  $N$ which is independent of the choice of $B$ and $\rho$.
\end{proposition}
\begin{proof}
  The fact that $N$ is a Lie algebra and that the Lie bracket
  defined above is independent of the choice of $B$ is shown
  in \cite[Prop. 2.2]{BKL}.
  We only need to verify Jacobi's identity
  \[ [x,[y,z]] + [y,[z,x]] + [z,[x,y]] =0 \]
  for homogeneous elements $x,y,z\in N$ when at least one of them
  belong to $(\radC q)^*$.
  Let $\xi,\zeta\in (\radC q)^*$ and let
  $x\in N_\alpha$ and $y\in N_\beta$ for some non-zero 
  $\alpha\in R$.
  We have two cases: 
  \begin{align*}
    [\xi,[\zeta,x]] + [\zeta,[x,\xi]] + [x,[\xi,\zeta]] &=
    \zeta(\rho(\alpha))[\xi,x] - \xi(\rho(\alpha))[\zeta,x]\\
    &= \zeta(\rho(\alpha))\xi(\rho(\alpha))x -
    \xi(\rho(\alpha))\zeta(\rho(\alpha))x
    = 0, 
  \end{align*}
  and since $[x,y] \in N_{\alpha + \beta}$, we have that
  \[
  [\xi,[x,y]] + [x,[y,\xi]] + [y,[\xi,x]] =
  \xi(\rho(\alpha+\beta))[x,y] - \xi(\rho(\beta))[x,y]
  + \xi(\rho(\alpha))[y,x] = 0.
  \]

  Finally, that the Lie algebra structure of $N$ is independent of
  the choice of $\rho$ can be shown exactly as in the proof of
  \cite[Prop. 5.3(i)]{BKL}.
\end{proof}

The Lie algebra $N$ is close to being an extended affine Lie
algebra, but it lacks a non-degenerate invariant bilinear form.
To achieve this we must consider the following quotient of $N$.

\begin{definition}
Let $E=E(q)$ be the factor Lie algebra of $N$ induced by the
canonical epimorphisms $N_\alpha\to\CC^n/\radC q$ for 
$\alpha\in R^0\setminus\set{0}$.
Thus the root spaces of $E$ are given by
\[ 
E_\alpha=\begin{cases}
  \CC e_{\alpha} &\text{if } \alpha\in R^\times,\\
  \CC^n/\radC q &\text{if } \alpha\in R^0\setminus\set{0},\\
  \CC^n\oplus(\radC q)^* &\text{if } \alpha=0.
\end{cases}
\]
\end{definition}

\begin{remark}
With some abuse of notation we also denote by $\pi_\alpha$ the
composition $\CC^n\xrightarrow{\pi_\alpha} N_\alpha \to \CC^n/\radC q$.
Thus rule (B1) induces the following rule on $E$: 
\[
[\pi_{\sigma}(v),\pi_{\tau}(w)] =
\epsilon(\sigma,\tau)q(v,w)\pi_{\sigma+\tau}(\sigma)
= \begin{cases}
  q(v,w)\pi_{0}(\sigma) &\text{if }\sigma+\tau=0,\\
  0 &\text{otherwise}
\end{cases}
\]
where $\sigma,\tau\in R^0$ and $v,w\in\CC^n$.
\end{remark}

\subsection{The invariant bilinear form on $E(q)$}

We define a bilinear form on $E$ by the following rules: 
\begin{align*}
  (e_{\alpha},-e_{-\alpha})&=1 \for \alpha\in R^\times,\\
  (\pi_{\sigma}(v),\pi_{-\sigma}(w))&=q(v,w) \for \sigma\in R^0\
  \and\ v,w\in \CC^n,\\ 
  (\xi,\pi_{0}(w))&= (\pi_0(w),\xi)=\xi(\rho(w)) \for w\in \CC^n,
\end{align*}
and zero in other cases. 
We need the following properties of $\epsilon(-,-)$ which follow
directly from the definition.

\begin{lemma}
  \label{lema:sign}
  Let $\alpha,\beta\in R^\times$ and $\sigma,\tau\in R^0$. Then
  \begin{itemize}
  \item[(i)] $\epsilon(\alpha,\alpha)=-1$ and
    $\epsilon(\sigma,\sigma)=1$. 
  \item[(ii)] $\epsilon(\sigma,\beta) = \epsilon(\beta,\sigma)$
    and $\epsilon(\sigma,\tau) = \epsilon(\tau,\sigma)$. 
  \item[(iii)]
    $\epsilon(\alpha+\sigma,\beta+\tau) =
    \epsilon(\alpha,\beta)\epsilon(\alpha,\tau)
    \epsilon(\sigma,\beta)\epsilon(\sigma,\tau)$. 
  \item[(iv)] If $\alpha+\beta\in R^\times$ then
    $\epsilon(\alpha,\beta) = -\epsilon(\beta,\alpha)$. 
  \item[(v)] If $\alpha+\beta\in R^0$ then
    $\epsilon(\alpha,\beta) = \epsilon(\beta,\alpha)$. 
  \end{itemize}
\end{lemma}

\begin{proposition}
  \label{prop:invariant_form}
  The symmetric bilinear form $\B:E\times E\to \CC$ is
  invariant and non-degenerate. 
\end{proposition}

\begin{proof} 
  We first show that $\B$ is non-degenerate. 
  Let $x\in E_\alpha$ be non-zero. 
  If $\alpha\in R^\times$ then $x=\lambda e_\alpha$ for
  some non-zero $\lambda\in\CC$.
  Hence $(x,-e_{-\alpha})=\lambda\neq 0$.
  Now suppose that $\alpha\in R^0$ is
  non-zero and $x=\pi_\alpha(v)$ for some $v\notin\radC q$. 
  Then there exists some non-zero $w\in \CC^n$ such that
  \[
  (x,\pi_{-\alpha}(w)) = (\pi_{\alpha}(v),\pi_{-\alpha}(w))=q(v,w)
  \neq 0.
  \]
  Finally, suppose that $\alpha=0$ and that $x\in \radC q$
  (resp. $x\in(\radC q)^*$).
  Then it is immediate that there exist some $y\in(\radC q)^*$
  (resp. $y\in \radC q$) such that $(x,y)\neq 0$.
  This shows that the form $\B$ is non-degenerate. 

  Now we prove that the form is invariant.
  Observe that it is
  sufficient to prove that 
  \[ (x,[y,z])=([x,y],z) \]
  for homogeneous elements $x,y,z\in E$ such that $\deg x+\deg
  y+\deg z=0$, where $\deg x=d$ for $x\in E_d$ as the form
  vanishes in other cases. 
  In order to prove this, we distinguish various cases.
  First, we treat the case where $x,y,z\notin (\radC q)^*$.
  Let $\alpha,\beta,\gamma\in R^\times$, $\sigma,\tau,\kappa\in
  R^0$ and $u,v,w\in\CC^n$.

  (i) Since $q(\sigma,-)=q(-,\tau)=0$ we have

  \begin{align*}
    (\pi_{\sigma}(u),[\pi_{\tau}(v),\pi_{\kappa}(w)]) =&
    (\pi_{\sigma}(u),\epsilon(\tau,\kappa)q(v,w)\pi_{\tau+\kappa}(\tau))\\
    =&\epsilon(\tau,\kappa)q(v,w)q(u,\tau)=0
  \end{align*}
    and
    \begin{align*}
    ([\pi_{\sigma}(u),\pi_{\tau}(v)],\pi_{\kappa}(w)) =&
    (\epsilon(\sigma,\tau)q(u,v)\pi_{\sigma+\tau}(\sigma),\pi_{\kappa}(w))\\
    =& \epsilon(\sigma,\tau)q(u,v)q(\sigma,w)=0.
    \end{align*}

  (ii) We have
  \[
  (\pi_{\sigma}(u),[e_\beta,e_\gamma]) =
  (\pi_{\sigma}(u),\epsilon(\beta,\gamma)\pi_{\beta+\gamma}(\beta))
  = \epsilon(\beta,\gamma)q(u,\beta).
  \]
  On the other hand, since $\sigma+\beta+\gamma=0$, we have
  \begin{align*}
    ([\pi_{\sigma}(u),e_\beta],e_\gamma) =&
    (\epsilon(\sigma,\beta)q(u,\beta)e_{\beta+\sigma},e_\gamma)
    = -\epsilon(\sigma,\beta)q(u,\beta)\\
    =& -\epsilon(-(\beta+\gamma),\beta)q(u,\beta)
    = -\epsilon(\beta,\beta)\epsilon(\gamma,\beta)q(u,\beta)\\
    =& \epsilon(\gamma,\beta)q(u,\beta).
  \end{align*}
  Since $\gamma+\beta\in R^0$, by Lemma \ref{lema:sign}(v) we have
  $\epsilon(\beta,\gamma)=\epsilon(\gamma,\beta)$.
  Thus we have the required equality in this case.

  (iii) We have
  \[
  (e_{\alpha},[e_\beta,e_{\gamma}]) =
  \epsilon(\beta,\gamma)q(\beta,\gamma)(e_\alpha,e_{-\alpha})
  = -\epsilon(\beta,\gamma)q(\beta,\gamma).
  \]
  On the other hand, since $\alpha+\beta+\gamma=0$, we have
  \begin{align*}
    ([e_{\alpha},e_\beta],e_{\gamma}) =&
    \epsilon(\alpha,\beta)q(\alpha,\beta)(e_{-\gamma},e_{\gamma})
    = -\epsilon(\alpha,\beta)q(\alpha,\beta)\\
    =& -\epsilon(-(\beta+\gamma),\beta)q(\alpha,\beta)
    = -\epsilon(\beta,\beta)\epsilon(\gamma,\beta)q(\alpha,\beta)\\
    =& \epsilon(\gamma,\beta)q(\alpha,\beta).
  \end{align*}
  Since $\beta+\gamma\in R^\times$, by Lemma \ref{lema:sign}(iv)
  we have $-\epsilon(\beta,\gamma)=\epsilon(\gamma,\beta)$.
  Finally, we have 
  \[ 
  q(\alpha,\beta)-q(\gamma,\beta) =
  q(\alpha,-\alpha-\gamma)-q(\gamma,-\alpha-\gamma)=0,
  \]
  since $\beta=-(\alpha+\gamma)$.
  Hence $q(\beta,\gamma)=q(\alpha,\beta)$.
  Thus we have the required equality in this case.

  We now consider the case when at least one of the elements
  involved lies in $(\radC q)^*$. We must show that
  \[
  (\xi,[x_{\alpha},x_{-\alpha}])=([\xi,x_{\alpha}],x_{-\alpha}) \] 
  for $\xi\in (\radC q)^*,\
  x_{\varepsilon\alpha}\in E(q)_{\varepsilon\alpha},\
  \alpha\in R$ and $\varepsilon=\pm1$. There are two non-trivial
  cases: 

  (i) Let $\xi\in (\radC q)^*,\ v,w\in \CC^n$ and let $\sigma \in
  R^0$ be non-zero. Then 
  \begin{align*}
    (\xi,[\pi_{\sigma}(v),\pi_{-\sigma}(w)]) &=
    q(v,w)(\xi,\pi_{0}(\sigma)) = q(v,w)\xi(\rho(\sigma))\text{ and}\\
    ([\xi,\pi_{\sigma}(v)],\pi_{-\sigma}(w)) &=
    \xi(\rho(\sigma))(\pi_{\sigma}(v),\pi_{-\sigma}(w)) = \xi(\rho(\sigma))q(v,w).
  \end{align*}
  (ii) Let $\xi\in (\radC q)^*$ and let $\alpha \in R^\times$. Then
  \begin{align*}
    (\xi,[e_{\alpha},e_{-\alpha}]) &=
    \epsilon(\alpha,\alpha)(\xi,\pi_{0}(\alpha)) =
    -\xi(\rho(\alpha))\text{ and}\\ 
    ([\xi,e_{\alpha}],e_{-\alpha}) &=
    \xi(\rho(\alpha))(e_{\alpha},e_{-\alpha}) =
    -\xi(\rho(\alpha)).\qedhere 
  \end{align*}
\end{proof}

The following is the main result of this section.

\begin{theorem}
  \label{thm:Eq-is-EALA}
  Let $\q$ be a connected non-negative unit form.
  Then the triple $(E(q),\B,\H)$ is an extended affine Lie algebra
  with $R(q)$ as root system.
\end{theorem}
\begin{proof}
  It is shown in Proposition \ref{prop:invariant_form} that there
  exists a non-degenerate invariant symmetric bilinear form
  $\B:E\times E\to\CC$.
  Let $\H=E_0$.
  Then the triple $(E,\B,\H)$ satisfies axioms (EA1), (EA2), (EA4)
  and (EA5) of extended affine Lie algebras,
  see Definition \ref{def:eala}.  
  It remains to show that $E$ satisfies axiom (EA3).
  For this, let $x_\alpha\in E_\alpha$ for some $\alpha\in
  R^\times$ and $x_\beta\in E_\beta$ for some $\beta\in R$.
  Then for each $n>0$ we have that 
  $(\ad x_\alpha)^n(x_\beta)\in E_{\beta+n\alpha}$. 
  But since $R$ is an extended affine root system it satisfies
  property (R6), hence $\beta+n\alpha\in R$ for only finitely many
  values of $n$.
  It follows that $E_{\beta+n\alpha}\neq\set{0}$ for only finitely
  many values of $n$, thus $\ad x_\alpha$ acts locally nilpotently
  on $E$.
  This shows that the triple $(E,\B,\H)$ satisfies axioms of
  extended affine Lie algebras.
\end{proof}

\section{A Kac-Moody-like Construction of $E(q)$}
\label{sec:kacmoody_construction}

We give an alternative construction of $E=E(q)$ as a quotient of the
Lie algebra $\tildeG=\tildeG(q)$ defined in \cite{BKL}. 
This is done by analogy to construction of the simply-laced affine
Kac-Moody algebras.
We recall the construction of $\tildeG$ and some
necessary results regarding the Lie algebra $\tildeG$.
We give the proof of Theorem \ref{thm:GKM} at the end
of the section. 

Let $\q$ be a connected non-negative unit form with symmetric
matrix $C\in\ZZ^{n\times n}$ and let $R:=R(q)$ be the root system
of $q$, see Definition-Proposition \ref{prop:Rq}.
We keep the notation of the previous section.
Let $F$ be the free Lie algebra on $3n$ generators
$\setP{f_{-i},h_i,f_{i}}{1\leqslant i\leqslant n}$ which are
homogeneous of degree $-c_{i},0,c_{i}$ respectively. 
Let $I(q)$ be the ideal of $F$ defined by the following
generalized Serre relations: 
\begin{enumerate}
\item[(G1)] $[h_i,h_j]=0$ for $i,j\in\set{1,\dots,n}$.
\item[(G2)] $[h_i,f_{\varepsilon j}] = \varepsilon
  C_{ij}f_{\varepsilon j}$ for $i,j\in\set{1,\dots,n}$ and 
  $\varepsilon=\pm 1$.  
\item[(G3)] $[f_{i},f_{-i}] = h_i$ for $i\in\set{1,\dots,n}$.
\item[(G$\infty$)] $[f_{\varepsilon_{t} i_{t}}, [f_{\varepsilon_{t-1} i_{t-1}},
  \dots,[f_{\varepsilon_{2} i_2},f_{\varepsilon_{1} i_{1}}]]=0$
  whenever
  $\sum_{k=1}^t\varepsilon_{k}c_{i_{k}}\notin R$ 
  for $\varepsilon_{k}=\pm 1$.  
\end{enumerate}
Let $\Gq=G(q):=F/I(q)$ and let $H$ be the abelian subalgebra of
$\Gq$ generated by $h_1,\ldots,h_n$. 

A \emph{monomial} in $\Gq$ is an element obtained from the
generators using the bracket only. 
Thus every monomial has a well defined degree. 
Let $\Gq_{\alpha}$ be
the space spanned by the monomials of degree $\alpha$. 
In general, the algebra $\Gq$ does not have a root space
decomposition, see \cite[Sec. 5]{BKL}. 
In order to achieve such a decomposition we need to extend
$F/I(q)$ by $(\radC q)^*$, the $\CC$-dual of $\radC q$. 
Let $\rho:\CC^n\rightarrow\radC q$ be any
projection. 
We define $\tildeG:=F/I(q)\oplus (\radC q)^*$ and we let 
$\H=H\oplus(\radC q)^*$. 
We define the following bracket rules which depend on the choice
of $\rho$:
\begin{align*}
  [\xi,x]&=-[x,\xi]=\xi(\rho(\alpha))x 
  \quad\text{for }
  \xi\in (\rad q)^*\text{ and } x\in \Gq_\alpha.\\ 
  [\xi,\zeta]&=0 \quad\text{for } \xi,\zeta\in (\radC q)^*.
\end{align*}

\begin{proposition}\cite[Lem. 5.2, Prop. 5.3(i)]{BKL}
  \label{prop:rho}
  Using the Lie algebra structure on $\Gq$, the rules above define
  a Lie algebra structure on $\tildeG$ which is independent of the
  choice of $\rho$. 
\end{proposition}

For $h=\sum_{i=1}^n \lambda_{i}h_{i}\in H\subset\tildeG$, define
$r(h)=\sum_{i=1}^n \lambda_{i}c_{i}\in\CC^n$. We define a bilinear
form $\B:\H\times H\rightarrow\CC^n$ by the rules 
\begin{align*}
  (h,h')&=r(h)^\top C r(h') \for h,h'\in H,\\
  (\xi,h)&=\xi(\rho(r(h))) \for \xi\in(\radC q)^*\ \and\ h\in H.
\end{align*}
We observe that this form is non-degenerate in the second
variable.
Indeed, if $r(h)\in\radC q$ is non-zero then there exists
$\xi\in(\radC q)^*$ such that $\xi(\rho(r(h))\neq0$. Also, for
$\alpha=\sum_{i=1}^n \lambda_i c_i\in \CC^n$ define
$h_\alpha:=\sum_{i=1}^n\lambda_i h_i\in H$ so as to have
$r(h_\alpha)=\alpha$. 

\begin{proposition}\cite[Prop. 5.3 (ii)]{BKL}
  The Lie algebra $\tildeG$ admits a root space decomposition,
  that is $\tildeG=\bigoplus_{\alpha\in
    R}\tildeG_\alpha$, 
  where 
  \[ 
  \tildeG_{\alpha}:=\{x\in\tildeG:[h,x]=(h,h_\alpha)x \forall
  h\in \H\} \for \alpha\in R. 
  \]
\end{proposition}

Moreover, we have the following result.

\begin{theorem}
  \label{thm:BKL-q-q'}
  \cite[Thm. 1.1]{BKL}
  Let $q$ and $q'$ be two equivalent connected non-negative unit
  forms.
  Then there exists an isomorphism of graded Lie algebras
  $\tilde{\varphi}:\tildeG(q)\to\tildeG(q')$ which maps the
  generators of $\tildeG(q)$ to the generators of $\tildeG(q')$.
\end{theorem}

We now proceed by analogy with the construction of the Kac-Moody
algebras, \cf \cite[Ch. 14]{Carter}.
The Lie algebra $\tildeG$ is an $\H^*$-graded $\H$-module, hence
it contains a unique maximal ideal $I$ such that $I\cap
\H=\{0\}$. 
The following proposition describes the relationship between
$\tildeG$ and $E$.

\begin{proposition}\label{prop:epi}
  Let $\q$ be a connected non-negative unit form and let $R(q)$ be
  the extended affine root system associated with $q$, 
  see Proposition \ref{prop:Rq}. 
  Then there is an epimorphism of graded Lie algebras 
  $\tildeG\rightarrow E$ which induces an isomorphism between 
  $\H\subset\tildeG$ and $\H\subset E$. 
\end{proposition}

\begin{proof}
  Let $h'_{i}=[e_{c_{i}},-e_{-c_{i}}]\in E$ for 
  $i\in\set{1,\dots,n}$. 
  It is readily verified that the elements
  $-e_{-c_{i}},h'_{i},e_{c_{i}}$ of $E$ satisfy relations
  (G1)-(G$\infty$) and that they generate
  $\CC^n\oplus\bigoplus_{\alpha\neq 0} E_\alpha$ as a Lie algebra,
  see Definition \ref{def:ears} and (R6) (compare also with the proof
  of \cite[Prop. 2.2]{BKL}). 
  Hence there is an epimorphism of graded Lie algebras
  $\Gq\rightarrow\CC^n\oplus\bigoplus_{\alpha\neq 0} E_\alpha$
  mapping $h_{i}$ to $h'_{i}$ and $f_{\varepsilon i}$ to 
  $\varepsilon e_{\varepsilon c_{i}}$ for $i\in\set{1,\dots,n}$ and
  $\varepsilon=\pm 1$ with the required properties. 
  This morphism extends in a natural way to an epimorphism of
  graded Lie algebras $\tildeG\rightarrow E$. 
  The last statement is immediate.
\end{proof}

We are ready to give the proof of Theorem \ref{thm:GKM}.

\begin{proof}[Proof of Theorem \ref{thm:GKM}]
  It suffices to show that $E$ has no non-zero ideal $I$ with
  $I\cap \H=\set{0}$, as in this case the epimorphism from
  Proposition \ref{prop:epi} induces an isomorphism of graded Lie
  algebras $\tilde{G}/I\rightarrow E$. Suppose $I$ is a
  non-zero ideal of $E$ with $I\cap \H=\set{0}$. Then $I$ is a
  graded ideal, that is
  \[ I=\bigoplus_{\alpha\in R} (E_\alpha\cap I). \]
  Let $x\in(E_\alpha\cap I)$ be a non-zero homogenous element of
  degree $\alpha\in R$. We distinguish two cases: 

  (i) If $\alpha\in R^\times$ then we may assume that
  $x=e_{\alpha}$. But
  $0\neq[e_{\alpha},-e_{-\alpha}]=\pi_{0}(\alpha)\in (\H\cap I)$ as
  $I$ is an ideal, which contradicts $I\cap \H=\{0\}$. 

  (ii) If $\alpha\in R^0$ is non-zero then $x=\pi_{\alpha}(v)$ for
  some $v\in \CC^n\setminus\radC q$. Hence there exists some
  $w\in\CC^n$ such that $q(v,w)\neq0$. Thus
  $0\neq[\pi_{\alpha}(v),\pi_{-\alpha}(w)]=q(v,w)\pi_{0}(\alpha)\in
  (\H\cap I)$, which again contradicts $I\cap \H=\{0\}$. This
  concludes the proof.\qedhere 
\end{proof}

Finally, we obtain the following result.

\begin{theorem}
  \label{thm:q-q'}
  Let $q$ and $q'$ be two connected non-negative unit
  forms.
  If $q$ and $q'$ are equivalent, then there is an
  isomorphism of graded Lie algebras $\varphi:E(q)\to E(q')$ which
  maps the generators of $E(q)$ to the generators of
  $E(q')$. Conversely, if $E(q)$ and $E(q')$ are isomorphic as graded
  Lie algebras, then $q$ and $q'$ are equivalent.
\end{theorem}

\begin{proof}
  Suppose that $q$ and $q'$ are equivalent. By Theorem
  \ref{thm:BKL-q-q'} there exists an isomorphism of graded Lie
  algebras $\tilde{\varphi}:\tildeG(q)\to\tildeG(q')$ which maps the
  generators of $\tildeG(q)$ to the generators of $\tildeG(q')$.  
  By Proposition \ref{prop:epi}, the isomorphism $\tilde{\varphi}$
  induces an isomorphism of graded Lie algebras $\varphi:E(q)\to
  E(q')$ which maps the generators of $E(q)$ to the generators of
  $E(q')$.

  Conversely, suppose that $E(q)$ and $E(q')$ are isomorphic as graded
  Lie algebras. In order to show that $q$ and $q'$ are equivalent, it
  is sufficient to show that they have the same corank and the same
  Dynkin type, see Theorem \ref{thm:Dynkin-types} and the remark after
  it. From the definition of $E(q)$ and $E(q')$ and the fact that
  $E(q)$ and $E(q')$ are isomorphic, it readily follows that $\corank
  q=0$ if and only if $\corank q'=0$. If $q$ and $q'$ do not have
  corank 0, let $0\neq\alpha\in R(q)^0$ and $0\neq\alpha'\in
  R(q')^0$. It follows that 
  \begin{align*}
      \corank q=&\tfrac{1}{2}(\dim_\CC \H(q)-\dim_\CC E(q)_\alpha)\\
      =&\tfrac{1}{2}(\dim_\CC \H(q')-\dim_\CC E(q')_{\alpha'})=\corank q'.
  \end{align*}
  In both cases, it follows that $q$ and $q'$ have $n$ variables,
  where
  \[
  n=\dim_\CC \H(q)-\corank q =\dim_\CC \H(q')-\corank q'.
  \]
  Let $\nu=\corank q=\corank q'$. To see that $q$ and $q'$ have the
  same Dynkin type, note that Theorem
  \ref{thm:Dynkin-types} implies that $q_\Delta^{-1}(1)$ induces a
  finite root system $\dot{R}\subseteq R(q)^\times$ of rank $n-\nu$ and
  of the Dynkin type of $q$. Similarly, $R(q')^\times$ contains a
  finite root system $\dot{R}'$ of rank $n-\nu$ and of the Dynkin type
  of $q'$. Finally, the graded isomorphism between $E(q)$ and $E(q')$
  implies that $\dot{R}$ and $\dot{R'}$ are isomorphic, as they
  correspond to isomorphic finite-dimensional Lie subalgebras of
  $E(q)$ and $E(q')$ respectively. Hence $q$ and
  $q'$ have the same Dynkin type. This concludes the proof.
\end{proof}


\begin{thebibliography}{1}
\bibitem{AABGP} B. Allison, S.Azam, S. Berman, Y. Gao and
  A. Pianzola: {\em Extended affine Lie algebras and their root
    systems} Mem. Amer. Math. Soc., 603 (1997). 
\bibitem{BdP} M. Barot, J. A. de la Pe\~na: {\em The Dynkin type
    of a non-negative unit form.} Expo. Math. 17 (1999), 339-348. 
\bibitem{BKL} M. Barot, D. Kussin and H. Lenzing: {\em The Lie
    algebra associated to a unit form.} J. Algebra 296 (2006),
  1-17.
\bibitem{Borcherds} R. Borcherds: {\em Vertex algebras, Kac-Moody
    algebras and the monster.} Proc. Nat. Acad. Sci. USA 83
  (1986), 3068-3071. 
\bibitem{Carter} R. Carter: {\em Lie algebras of finite and affine
    type.} Cambridge University Press, Cambridge, (2005). 
\bibitem{ME} G. Jasso: {\em \'Algebras de Lie de tipo af\'{i}n
    extendido y formas cuadr\'aticas.} Undergraduate Thesis,
  Universidad Nacional Aut\'onoma de M\'exico (2010). 
\end{thebibliography}
\end{document}